\documentclass[a4paper,10pt,notitlepage]{amsart}
\usepackage[english]{babel}
\usepackage[latin1]{inputenc}
\usepackage[T1]{fontenc}
\usepackage{amssymb,amsthm,amsmath} 
\usepackage{mathtools}
\usepackage{mathrsfs}
\usepackage{enumitem}
\usepackage{geometry}
\usepackage{hyperref}
\theoremstyle{plain} 
\newtheorem{theorem}{Theorem}[section]
\newtheorem{lemma}{Lemma}[section]

\theoremstyle{definition}

\theoremstyle{remark}
\newtheorem{remark}{Remark}[section]

\DeclarePairedDelimiter{\abs}{\lvert}{\rvert} 
\DeclarePairedDelimiter{\norm}{\lVert}{\rVert}  

\DeclareMathOperator{\divergenza}{div}
\renewcommand{\div}{\divergenza}

\DeclareMathOperator{\sgn}{sgn}

\DeclareMathOperator{\dx}{dx}

\newcommand{\R}{\mathbb{R}}
\newcommand{\C}{\mathbb{C}}
\newcommand{\D}{\mathscr{D}}
\newcommand{\normeq}[1]{{\left\vert\kern-0.25ex\left\vert\kern-0.25ex\left\vert #1 
    \right\vert\kern-0.25ex\right\vert\kern-0.25ex\right\vert}}

\newenvironment{system}%
{\left\lbrace\begin{array}{@{}l@{}}}%
{\end{array}\right.}

\title[Spectral properties of Lamé operators]{\textbf{Uniform resolvent estimates and absence of eigenvalues for Lamé operators with potentials}}

\author{Lucrezia Cossetti}
\address{Lucrezia Cossetti: Dipartimento di Matematica, Sapienza Università di Roma, P.le A. Moro 5, 00185, Roma}
\email{cossetti@mat.uniroma1.it}

\subjclass[2010]{35J47, 47A10}

\keywords{Lamé operator, Helmholtz equation, Spectral Theory, Resolvent estimates}

\begin{document}

\date{\today}

%%%%%%%%%%%%%%%%%%%%%%%%%%%%%%%%%%%%%%%%%%%%%%%%%%%%%%%%%%%%%%%%%%%%%%%%%%%%%%%%%%%%%%%%%%%%%%%%%%%%%%%%%%%%%%%%%%%%%%%%%%%%%%%%%%%%%%%%%%

\begin{abstract}
We consider the $0$-order perturbed Lamé operator $-\Delta^\ast + V(x)$. 
It is well known that if one considers the free case, namely $V=0,$ the spectrum of $-\Delta^\ast$ is purely continuous and coincides with the non-negative semi-axis. 
The first purpose of the paper is to show that, at least in part, this spectral property is preserved in the perturbed setting. Precisely, developing a suitable multipliers technique, we will prove the absence of point spectrum for Lamé operator with potentials which satisfy a variational inequality with suitable small constant. We stress that our result also covers complex-valued perturbation terms. Moreover the techniques used to prove the absence of eigenvalues enable us to provide uniform resolvent estimates for the perturbed operator under the same assumptions about $V$. 
\end{abstract}

%%%%%%%%%%%%%%%%%%%%%%%%%%%%%%%%%%%%%%%%%%%%%%%%%%%%%%%%%%%%%%%%%%%%%%%%%%%%%%%%%%%%%%%%%%%%%%%%%%%%%%%%%%%%%%%%%%%%%%%%%%%%%%%%%%%%%%%%%%

\maketitle

%%%%%%%%%%%%%%%%%%%%%%%%%%%%%%%%%%%%%%%%%%%%%%%%%%%%%%%%%%%%%%%%%%%%%%%%%%%%%%%%%%%%%%%%%%%%%%%%%%%%%%%%%%%%%%%%%%%%%%%%%%%%%%%%%%%%%%%%%%

\section{Introduction}
This paper is concerned with operators of the form
\begin{equation*}
	-\Delta^\ast + V(x)
\end{equation*}  
acting on the Hilbert space $[L^2(\R^d)]^d$ that is the Hilbert space of the vector fields with components in $L^2(\R^d)$, where $-\Delta^\ast$ denotes the Lamé operator of elasticity, that is a linear symmetric differential operator of second order that acts on smooth $L^2$ vector fields $u$ on $\R^d,$ for example $[C^\infty_c(\R^d)]^d,$ in this way:
\begin{equation*}
	-\Delta^\ast u:= -\mu(x) (\Delta u_1, \Delta u_2, \dots, \Delta u_d) 
									 - (\lambda(x) + \mu(x)) \nabla \div(u_1,u_2, \dots, u_d).
\end{equation*}
Let us recall that in order to obtain the positivity of the quadratic form associated with this operator, we assume
\begin{equation}
	\label{positive_condition}
		\mu>0, \, 
		\lambda> - \frac{2}{3}\mu;  
\end{equation}
the functions $\lambda$ and $\mu$ are the so called Lamé's coefficients.
Moreover $V(x)$ is a notation for the multiplication operator by the function $V(x)$ that, in our setting, we are assuming to be complex-valued measurable function on $\R^d.$  
We stress that, under this assumption on $V,$ the context we are working in is a non self-adjoint setting. 

The theory of non self-adjoint operators is very much less unified than the self-adjoint one. 
Nevertheless, recently there is a growing interest in this context, indeed there is an increasing number of problems, particularly in physics, which require the analysis of non self-adjoint operators.

Just to quote a pair of them, first let us mention the very recent problem to investigate on the distribution, in the complex plane, of eigenvalues of the non self-adjoint Schr\"odinger operators. Roughly speaking this is the problem to find the correct analogue to the Lieb-Thirring inequalities (cf.~\cite{L_T}) for non self-adjoint operators, see, for example,~\cite{A_A_D, L_S, D_H_K-II, F_S}.

An other context, in which dealing with non self-adjoint operator turns out to be useful, is the study of complex resonances of self-adjoint Schr\"odinger operator (cf.~\cite{A_A_D, B_O}). 
If one uses the techniques of complex scaling, then these resonances turn into eigenvalues of associated (unitarily equivalent) non self-adjoint Schr\"odinger operator.

Fanelli, Krej\v{c}i\v{r}\'ik and Vega, in a very recent work~\cite{F_K_V}, that primarily motivates our paper, improve the state of the art in the picture of spectral properties for non self-adjoint Schr\"odinger operators.

Precisely, using a generalized version of the Birman-Schwinger principle, they proved that in three dimensions the spectrum of Schr\"odinger operator $H$ is purely continuous and coincides with the non-negative semi-axis for all complex-valued potential satisfying the following form subordinated smallness condition:
\begin{equation}
\label{Laplace_smallness_condition}
	\exists\, a<1,\quad \forall \psi \in H^1(\R^d), 
	\qquad \int_{\R^d} \abs{V}\abs{\psi}^2 \leq a \int_{\R^d} \abs{\nabla \psi}^2.  
\end{equation}    
The relevance of this condition relies on the fact that, first it seems to be new also in the self-adjoint setting, moreover if one is interested in the less strong result about the mere absence of eigenvalues, this is a weaker condition than the more classical ones (for example, the belonging to the Rollnick class) and, at the same time, it seems to recover potentials that are not typically covered by previous works treating those arguments.

In addition to this result,  with a completely alternative approach, which substantially exploits a suitable development of the multipliers technique, they prove the absence of point spectrum in all dimensions, under a stronger hypothesis on the potential then the previous one.

Our first purpose goes in this direction. Precisely we want to investigate if some spectral stability properties are preserved under small perturbations of the operator we are dealing with, i.e. the Lamé operator.

The Helmholtz decomposition strongly comes into play. In fact, making use of this tool, which is a standard way to decompose smooth vector fields into a sum of a divergence free vector field and a gradient, we can explicitly see the connection between the Lamé operator and the Laplacian. Indeed, it is very easy to see that, for any $u=u_p+u_S,$ the operator $-\Delta^\ast$ acts on $u$ in this way:
\begin{equation*}
	-\Delta^\ast u= -\mu(x) \Delta u_S -(\lambda(x) +2\mu(x)) \Delta u_P, 
\end{equation*} 
where the component $u_S$ is the divergence free vector field and the component $u_P$ is the gradient.

The relation between the two operators, that the previous identity have highlighted, has played a fundamental role in order to motivate this paper. In fact, our purpose is to prove a result that is the counterpart of that Fanelli, Krej\v{c}i\v{r}\'ik and Vega in~\cite{F_K_V} have proved for the Laplacian and we are going to do this following a strategy that is the natural generalization of their one.  

Precisely we will prove the following result, which substantially guarantees that, in $d\geq 3,$ the absence of eigenvalues is preserved under suitable smallness and complex perturbation of the operator.

\begin{theorem} \label{thm:absence_of_eigenvalues}
	Let $d\geq 3.$ Assume that $\lambda, \mu \in \R$ satisfy~\eqref{positive_condition}
	and that $V\colon \R^d \to \C$ is such that 
	\begin{equation} \label{smallness_condition_l}
		\forall\, \psi\in H^1(\R^d), \qquad \int_{\R^d} \abs{x}^2 \abs{V}^2 \abs{\psi}^2\leq \Lambda^2 \int_{\R^d} \abs{\nabla \psi}^2,  
	\end{equation} 
	where $\Lambda$ satisfies  
	\begin{equation} \label{Lambda_condition}
		\frac{4\Lambda}{\min\{\mu, \lambda + 2\mu\}} \frac{d(2d-3)}{d-2} (C+1)  
		+\frac{8 \Lambda^\frac{3}{2}}{\sqrt{\min\{\mu, \lambda + 2\mu\}}^3} \frac{d^\frac{3}{2}}{\sqrt{d-2}} (C+1)^\frac{3}{2}<1,
	\end{equation}
	and where $C>0$ is a suitable constant.
	Then $\sigma_p(-\Delta^\ast + V)=\varnothing.$ 
\end{theorem}
\begin{remark}
The constant $C>0$ that is in the statement of the above theorem is the which one that will appear in the elliptic regularity Lemma~\ref{elliptic_regularity} we will state below.
\end{remark}
\begin{remark}
Let us stress that for the physical case $d=3$ (that is covered by our result), although we only prove the absence of the point spectrum, we expect that an analogous to the stronger result in~\cite{F_K_V} (Theorem 1) holds but this will be matter of other investigation.
We just remark that it is not obvious how to refine the proof in~\cite{F_K_V} and to develop this in our context. This, primarily, is due to the fact that the Birman-Schwinger operator associated to Lamé, instead of Laplace operator, has a ``lack of symmetry'' in his structure which does not allow to apply directly the strategy under the Birman-Schwinger principle.     
\end{remark}

\begin{remark}
We underline that, actually, the most natural generalization of the condition~\eqref{Laplace_smallness_condition} about $V,$ which appears in~\cite{F_K_V}, is not~\eqref{smallness_condition_l} but the following one:
\begin{equation} \label{smallness_condition_a}
		\exists\, a<\min\{\mu, \lambda + 2\mu\}, \qquad \forall\, \psi \in H^1(\R^d), 
		\qquad \int_{\R^d} \abs{V}\abs{\psi}^2 \leq a \int_{\R^d} \abs{\nabla \psi}^2.  
\end{equation}

Moreover, it's not difficult to see that the condition~\eqref{smallness_condition_l}, which is in our result, is stronger than~\eqref{smallness_condition_a} in fact, assuming that~\eqref{smallness_condition_l} holds and making use of the classical Hardy inequality
\begin{equation} \label{classical_Hardy}
	\forall\, \psi\in H^1(\R^d), 
	\qquad \int_{\R^d} \frac{\abs{\psi(x)}^2}{\abs{x}^2} \, dx \leq \frac{4}{(d-2)^2} \int_{\R^d} \abs{\nabla \psi}^2,
\end{equation} 
one has
\begin{equation*}
	\forall\, \psi\in H^1(\R^d) 
	\qquad \int_{\R^d} \abs{V}\abs{\psi}^2 \leq \norm{x V \psi} \norm*{\frac{\psi}{\abs{x}}}\leq \frac{2\Lambda}{d-2}\norm{\nabla \psi}^2,
\end{equation*}
as a consequence of the restriction~\eqref{Lambda_condition} we have $\Lambda< \min \{ \mu, \lambda + 2\mu\} \frac{d-2}{2}$ then $2\Lambda/(d-2)< \min \{ \mu, \lambda + 2\mu\},$ hence~\eqref{smallness_condition_a} holds.
\end{remark}

\begin{remark}
Let us observe that, in our theorem, we are assuming that $\lambda$ and $\mu$ are constants, on the other hand, an interesting open problem, that is not object of this paper, is concerned with the validity of the similar results in the variable-coefficients setting, precisely $\mu=\mu(x)$ and $\lambda=\lambda(x).$
\end{remark}

As a further application of the multipliers technique we have developed to prove Theorem~\ref{thm:absence_of_eigenvalues}, we are also able to perform uniform resolvent estimates for the operator $-\Delta^\ast + V,$  which generalize the ones obtained, for the Helmholtz equation, by Barcel\'o, Vega and Zubeldia in~\cite{B_V_Z}.  
 
Precisely, in the last section of the paper we consider the eigenvalues perturbed equation
\begin{equation} \label{eigenvalue_complete}
	\Delta^\ast u - Vu + ku=f,
\end{equation}
where $k=k_1 + i k_2$ is any complex constant and $f\colon \R^d \to \R^d$ is a measurable function and we will prove, for solution of~\eqref{eigenvalue_complete}, the following result 
\begin{theorem} \label{thm:uniform_resolvent_estimate}
		Let $d\geq 3,$ $\norm{\abs{\,\cdot\,}f}<\infty$ and assume that $V$ satisfies~\eqref{smallness_condition_l}. Then, there exist $c>0$ independent of $k$ and $f$ such that for any solution $u\in [H^1(\R^d)]^d$ of the equation~\eqref{eigenvalue_complete} one has
\begin{equation} \label{uniform_resolvent_estimate}
	\norm{\abs{x}^{-1} u}\leq c \norm{\abs{x} f}.
\end{equation}
\begin{remark}
	We remark that the estimate~\eqref{uniform_resolvent_estimate} is already proved in~\cite{B_F_R_V_V}. On the other hand our integral-smallness assumption on the potential is weaker than the one required in that work. Indeed, to be more precise, the authors provide the uniform resolvent estimate~\eqref{uniform_resolvent_estimate} for the equation
	\begin{equation*}
		\Delta^\ast u + k u -\delta V u=f,
	\end{equation*}
	assuming that $\norm{\abs{x}^2 V}_{L^\infty}<\infty$ and  $\abs{\delta}$ sufficiently small to let the perturbation argument work.  
\end{remark}
\end{theorem}

Actually, in order to prove Theorem~\ref{thm:uniform_resolvent_estimate} we establish the following stronger result, which shows that a priori estimates for solutions of~\eqref{eigenvalue_complete} hold.  
\begin{theorem} \label{thm:a_priori_estimates}
	Let $d\geq 3,$ $\norm{\abs{\,\cdot\,}f}<\infty$ and assume that $V$ satisfies~\eqref{smallness_condition_l}. Then, there exist $c>0$ independent of $k$ and $f$ such that for any solution $u\in [H^1(\R^d)]^d$ of the equation~\eqref{eigenvalue_complete} one has
	\begin{itemize}
		\item for $\abs{k_2}\leq k_1$
			\begin{equation} \label{a_priori_u_S_minus_u_P_minus}
				\norm{\nabla u_S^-}\leq c \norm{\abs{x} f}, \qquad \text{and} \qquad \norm{\nabla u_P^-}\leq c \norm{\abs{x} f},
			\end{equation}
		\item for $\abs{k_2}> k_1$
			\begin{equation} \label{a_priori_u}
				\norm{\nabla u}\leq c \norm{\abs{x} f}.
			\end{equation}
	\end{itemize}
	\begin{remark}
	The vector fields $u_S^-$ and $u_P^-,$ which appear in the statement, are defined in~\eqref{def:u_S_minus} and~\eqref{def:u_P_minus} respectively.
	\end{remark}
\end{theorem}

From this, as a straightforward corollary, we easily obtain Theorem~\ref{thm:uniform_resolvent_estimate}.

%%%%%%%%%%%%%%%%%%%%%%%%%%%%%%%%%%%%%%%%%%%%%%%%%%%%%%%%%%%%%%%%%%%%%%%%%%%%%%%%%%%%%%%%%%%%%%%%%%%%%%%%%%%%%%%%%%%%%%%%%%%%%%%%%%%%%%%%%%%

\section{Preliminaries}

We devote this preliminary section to recall some very well known facts about the Lamé operator which, actually, are interesting in their own sake.   

First of all we want to give a rigorous meaning to the Lamé operator as a self-adjoint operator, i.e.  we want to build the self-adjoint extension of the operator $-\Delta^\ast;$ in order to do that we proceed using a quadratic form approach.

Let us introduce the quadratic form associated with the operator $-\Delta^{\ast},$
\begin{equation*}
	Q_0[u]=\int_{\R^d}\, q_0[u]\, dx,
\end{equation*}   
where
\begin{equation*}
	q_0[u]= \lambda \left(\sum_{i=1}^d \frac{\partial u_i}{\partial x_i} \right)^2 + 
					\frac{\mu}{2} \sum_{i, j=1}^d \left(\frac{\partial u_i}{\partial x_j} + \frac{\partial u_j}{\partial x_i} \right)^2,
					\qquad u\in [C^\infty_c(\R^d)]^d.
\end{equation*}
In order that $q_0[u],$ and, thus, $-\Delta^\ast$ is positive, we assume for $\mu$ and $\lambda$ the condition~\eqref{positive_condition}.

We recall that, since our form $Q_0$ is associated with a densely defined positive and symmetric operator, this form is closable.
Moreover it is easy to see that the domain of the closure of $Q_0$ is the Sobolev space of $H^1$- vector fields.

As $\bar{Q_0}$ is a densely defined lower semi-bounded (actually positive) close form on an Hilbert space, then there is a canonical way to built from it a distinguished self-adjoint extension, called Friedrichs extension, of the symmetric operator $-\Delta^\ast,$ that is the self-adjoint operator we are looking for and that we, again, write as $-\Delta^\ast.$   

Now our purpose is to understand the action of $-\Delta^\ast$ on smooth vector fields. In order to do that we are going to make use of the well known \emph{Helmholtz decomposition}, which is a standard way to decompose a vector field into a sum of a gradient and a divergence free vector field.  
To be more precise, we have that every smooth vector field $u\colon \R^d \to \R^d,$ sufficiently rapidly decaying at infinity, can be uniquely decomposed as 
	\begin{equation*}
		u=u_S + u_P,
	\end{equation*}
	where $\div u_S=0$ and $u_P=\nabla \varphi,$ for some smooth scalar function $\varphi.$
	
\begin{remark}
Let us remark that from the previous assumptions immediately follows  that $u_S$ and $u_P$ are $L^2$- orthogonal. In view of our aims, actually one can also prove that $u_S$ and $u_P$ are $H^1$- orthogonal. 
\end{remark}

Using the Helmholtz decomposition, a straightforward computation shows that for any $u=u_S + u_P,$ the operator $- \Delta^\ast$ acts on $u$ in this way
\begin{equation*}
	-\Delta^\ast u= -\mu \Delta u_S -(\lambda +2 \mu) \Delta u_P,
\end{equation*}   
where, again, $u_S$ is a divergence free vector field and $u_P$ is a gradient. 

The quadratic form associated with the operator $- \Delta^\ast,$ explicitly written in the Helmholtz decomposition, is
\begin{equation*}
	Q_0[u]=\int_{\R^d}\, q_0[u]\, dx,
\end{equation*}
with
\begin{equation*}
	q_0[u]= \mu \abs{\nabla u_S}^2 + 
					(\lambda + 2\mu) \abs{\nabla u_P}^2 
	\qquad \text{and} \qquad \D(Q_0)=[H^1(\R^d)]^d,
\end{equation*}
where $\abs{\nabla v}^2,$ when $v=(v_1, v_2, \dots, v_d)$ is a vector field, denotes $\sum_{j=1}^d \abs{\nabla v_j}^2.$

We observe that here, with an abuse of notation, we have used the same symbol $Q_0$ for the quadratic form associated with $-\Delta^\ast,$ both when its action is written explicitly using the Helmholtz decomposition and when the operator is defined in its classical way.

Let us note that, assuming $\mu>0$ and $\lambda + 2\mu>0,$ then the quadratic form turns out to be positive, i.e. $-\Delta^\ast$ elliptic.  

These conditions on the Lamé's coefficients are weaker than~\eqref{positive_condition} which we have stated above to make positive the quadratic form.

Now we are in position to consider the perturbed operator
\begin{equation*}
	-\Delta^\ast + V(x),
\end{equation*} 
where $V\colon \R^d \to \C$ is the perturbation term.

Clearly, in the Helmholtz decomposition, this operator acts on a smooth vector fields $u$ in this way
\begin{equation*}
	(-\Delta^\ast + V)u= -\mu \Delta u_S -(\lambda +2 \mu) \Delta u_P + V u. 
\end{equation*} 
The corresponding perturbed quadratic form associated with this operator is
\begin{equation*}
	Q_\text{pert}[u]= Q_0[u] + Q_V[u]
									= Q_0[u] + \int_{\R^d}\, q_V[u] \, dx,
\end{equation*}
where 
\begin{equation*}
	q_V[u]= V\abs{u}^2 
	\qquad \text{and} \qquad \D(Q_V)= \left \{u \in [L^2(\R^d)]^d : \int_{\R^d} \abs{V} \abs{u}^2< \infty \right\}.
\end{equation*}
We assume the smallness condition~\eqref{smallness_condition_a} about $V.$
It's not difficult to see that, as a consequence of the constrictions on $a$, $Q_V$ is relatively bounded with respect to $Q_0$ with bound less than one.

Let us suppose, for a moment, that our potential $V$ is real-valued. As a consequence, the sesquilinear form, associated with the quadratic form $Q_V,$ is symmetric.
By virtue of these remarks, we are able to build from $Q_\text{pert}$ an associated self-adjoint operator on $[L^2(\R^d)]^d$ exploiting the well known forms counterpart of the Kato-Rellich perturbation result for operators, namely the KLMN theorem (see for example~\cite{R_S-II}, Thm X.$17$, or~\cite{Schmudgen}, Thm $10.21$). 

If one is dealing with complex-valued potentials, as our setting, instead of real-valued ones, the scenario turns out to be quite different.
In fact, assuming now that $V$ is a complex-valued potential, the sesquilinear form $Q_V$ is no more symmetric and, as a consequence, we clearly can't expect to be able to build from $Q_\text{pert}$ a self-adjoint extension of $-\Delta^\ast + V$.
Nevertheless, even though we are dealing with non symmetric forms, we can obtain useful information about the operator $-\Delta^\ast + V$ by exploiting the theory about sectorial forms (resp. operators).
Precisely we can use the representation theorem (see~\cite{Kato}, Thm. VI$.2.1$) to build an m-sectorial operator from a densely defined, sectorial and closed form.  

%%%%%%%%%%%%%%%%%%%%%%%%%%%%%%%%%%%%%%%%%%%%%%%%%%%%%%%%%%%%%%%%%%%%%%%%%%%%%%%%%%%%%%%%%%%%%%%%%%%%%%%%%%%%%%%%%%%%%%%%%%%%%%%%%%%%%%%%%%%

\section{Absence of eigenvalues: proof of Theorem~\ref{thm:absence_of_eigenvalues}}
We devote this section to the proof of the Theorem~\ref{thm:absence_of_eigenvalues} we stated in introduction.

We recall that our strategy wants to be build in analogy to that one in the recent work~\cite{F_K_V} of Fanelli, Krej\v{c}i\v{r}\'ic and Vega, who established the analogous result for the Laplace operator.  

First of all, to this end, starting from the eigenvalue equation associated with the perturbed Laplacian, they provided three integral identities which had a crucial role in the proof of their main result; in order to do that they re-adapted to a non self-adjoint setting the standard technique of Morawetz multipliers. 
This tool was introduced in~\cite{Morawetz} for the Klein-Gordon equation and then it was developed in several other contexts. For example, about the Helmholtz equation's framework, let us mention the seminal works of Perthame and Vega~\cite{P_V},~\cite{P_V-II} which concern the purely electric case and then~\cite{F_V,Fanelli,Zubeldia,B_F_R_V,B_V_Z,Zubeldia_II}, which extend the technique in an electromagnetic setting. We should also quote~\cite{C_D_L} for an adaptation of multiplier metod on exterior domains.      

Now we are in position to begin the proof of our result.

The eigenvalue problem for the perturbed Lamé operator is
\begin{equation} \label{eigenvalue_equation}
	\Delta^\ast u + k u=V u,
\end{equation} 
where $k$ is any complex constant (throughout the paper we will denote by $k_1=\Re k$ and by $k_2=\Im k$).

Just to simplify the notations, we start assuming that $u$ is a solution of this more general problem
\begin{equation} \label{f_equation} 
	\Delta^\ast u + k u=f,
\end{equation}
where $f\colon \R^d \to \C^d$ is a measurable function. 

Clearly, we can identify the problem~\eqref{f_equation} with~\eqref{eigenvalue_equation} by setting $f=V u.$
 
As we said above, the Helmholtz decomposition has been a fundamental tool for our purposes, according to this, writing $u=u_S + u_P$ and $f=f_S + f_P,$ the eigenvalues problem~\eqref{f_equation} associated to the Lamé operator can be re-written as
\begin{equation*} 
	\mu \Delta u_S + (\lambda + 2\mu) \Delta u_P + k u_S + k u_P=f_S + f_P,
\end{equation*}
where, again, the $S$ component is the divergence free vector field and the $P$ component is the gradient. 

Let us observe that the equation written in this form is very far to be easy to handle, in fact the two components has the same frequency of oscillation $k$ but different speed of propagation $\mu$ and $\lambda + 2\mu$ respectively, and therefore the first attempt one would like to try is splitting the previous equation into a system of two decoupling equations involving separately the two components $u_S$ and $u_P.$ 
This attempt is going to work indeed, as a consequence of the $H^1$-orthogonality of the $S$ and $P$ components of the Helmholtz decomposition, we are allowed to reduce our ``intertwining'' equation into a system of two decoupling equations
\begin{equation} \label{decoupling}   
	\begin{system}
		\mu \Delta u_S + k u_S=f_S \\
		(\lambda + 2\mu) \Delta u_P + k u_P=f_P.
	\end{system}
\end{equation}
At this point, the next step is, in some sense, obliged. In fact the most natural way to proceed is to prove, separately for the two equations, the analogue estimates which Fanelli, Krej\v ci\v r\'ic and Vega have formerly proved in their paper, clearly provided suitable changes due to the presence of coefficients of the Laplace operators.
 
As a starting point, we consider the weak formulation of~\eqref{decoupling}
\begin{equation} \label{decoupling_weak}
	\forall\, v\in [H^1(\R^d)]^d, \qquad
						\begin{system}
							-\mu (\nabla v, \nabla u_S) + k(v, u_S)=(v,f_S)\\
							-(\lambda + 2\mu) (\nabla v, \nabla u_P) + k(v, u_P)=(v, f_P).
						\end{system}
\end{equation}
Following~\cite{B_V_Z} we divide the proof of our result into two cases depending on the relation between real and imaginary part of the eigenvalue $k$: $\abs{k_2}\leq k_1$ and $\abs{k_2}>k_1.$ 

Let us start by the more technical case $\abs{k_2}\leq k_1$

\begin{description}[style=unboxed,leftmargin=0cm]
\item[Case $\abs{k_2}\leq k_1.$]
Following~\cite{F_K_V} we want to obtain from~\eqref{decoupling_weak}, with suitable algebraic manipulations, three fundamental identities involving only $u_S$ and the analogous identities involving only $u_P.$ 
The first two will be obtained using, as a test function, a symmetric multiplier, the third using an anti-symmetric one.
Throughout the paper we are going to omit arguments of integrated function and the differential $\dx$

Clearly, in order to make rigorous the following integrations by parts we assume $u$ and $f$ sufficiently regular, so that we will get our result by a standard density argument.

We begin with the manipulation of the first of~\eqref{decoupling_weak}. 
Let $\phi_1, \phi_2, \phi_3 \colon \R^d \to \R$ be three sufficiently smooth functions. 

\begin{remark}
Let us underline that only $f$ is a complex-vector-valued function. 
\end{remark}

Choosing as a test function $v:= \phi_1 u_S$ in the first of~\eqref{decoupling_weak}, integrating by parts and taking the real part of the resulting identity, one gets
\begin{equation} \label{first_u_S}
	k_1\int_{\R^d} \phi_1 \abs{u_S}^2 - \mu \int_{\R^d} \phi_1 \abs{\nabla u_S}^2 + \frac{\mu}{2} \int_{\R^d} \Delta \phi_1 \abs{u_S}^2
	= \Re \int_{\R^d} \phi_1 u_S f_S.
\end{equation}
Choosing the same multiplier of before, $v:= \phi_2 u_S$ in the first of~\eqref{decoupling_weak} and taking, this time, the imaginary part of the resulting identity, we arrive at the second identity
\begin{equation} \label{second_u_S}
	k_2 \int_{\R^d} \phi_2 \abs{u_S}^2 = \Im \int_{\R^d} \phi_2 u_S f_S.
\end{equation}
In the end, choosing the antisymmetric multiplier $v:=[\Delta, \phi_3]u_S=2 \nabla \phi_3\cdot \nabla u_S + \Delta \phi_3 u_S$ in the first of~\eqref{decoupling_weak}, integrating by parts and taking the real part of the resulting identity, one obtains
\begin{equation} \label{third_u_S}
	\mu \int_{\R^d} \nabla u_S \cdot D^2 \phi_3\cdot \nabla u_S 
	- \frac{\mu}{4} \int_{\R^d} \Delta^2 \phi_3 \abs{u_S}^2 
	= -\Re \int_{\R^d} \nabla \phi_3 \cdot \nabla u_S f_S - \frac{1}{2} \Re \int_{\R^d} \Delta \phi_3 u_S f_S,
\end{equation}  
where the symbol $[\cdot,\cdot]$ is used, as usual, for the commutator and  $D^2,$ $\Delta^2$  denote, respectively, the Hessian matrix and  the bi-Laplacian.
 
Following the same strategy used to obtain the three above identities for $u_S,$ we have the analogous identities involving only $u_P.$

Choosing the multiplier $v:=\phi_1 u_P$ in the second of~\eqref{decoupling_weak}, one gets
\begin{equation} \label{first_u_P}
	k_1\int_{\R^d} \phi_1 \abs{u_P}^2 - (\lambda + 2\mu) \int_{\R^d} \phi_1 \abs{\nabla u_P}^2 
	+ \frac{\lambda + 2\mu}{2} \int_{\R^d} \Delta \phi_1 \abs{u_P}^2
	= \Re \int_{\R^d} \phi_1 u_P f_P.
\end{equation}
Choosing $v:=\phi_2 u_P$ in the second of~\eqref{decoupling_weak}, one obtains 
\begin{equation} \label{second_u_P}
	k_2 \int_{\R^d} \phi_2 \abs{u_P}^2 = \Im \int_{\R^d} \phi_2 u_P f_P.
\end{equation}
Finally, choosing again the antisymmetric multiplier $v:=[\Delta,\phi_3]u_P,$ we have 
\begin{equation} \label{third_u_P}
	(\lambda + 2\mu) \int_{\R^d} \nabla u_P \cdot D^2 \phi_3\cdot \nabla u_P 
		- \frac{\lambda + 2\mu}{4} \int_{\R^d} \Delta^2 \phi_3 \abs{u_P}^2 
		= -\Re \int_{\R^d} \nabla \phi_3 \cdot \nabla u_P f_P - \frac{1}{2} \Re \int_{\R^d} \Delta \phi_3 u_P f_P.
\end{equation} 
We devote the two following subsections to handle, separately, the three identities for $u_S$ and for $u_P.$ 
From now on, we assume that $\phi_1, \phi_2$ and $\phi_3$ are radial, namely we assume that exist smooth functions $\psi_1, \psi_2, \psi_3\colon [0,\infty) \to \R$ which depend only of $\abs{x}$ such that $\phi_i(x)=\psi_i(\abs{x})$ for all $x\in \R^d$ and $i\in\{1,2,3\}.$\\
A straightforward computation shows that the following relations hold:
\begin{equation*}
	\nabla \phi_i(x)=\psi_i'(\abs{x}) \frac{x}{\abs{x}}, \quad \Delta \phi_i(x)=\psi_i''(\abs{x}) + \psi_i'(\abs{x})\frac{d-1}{\abs{x}},\quad
	D^2 \phi_i(x)=\psi_i''(\abs{x}) \frac{xx}{\abs{x}^2} + \frac{\psi_i'(\abs{x})}{\abs{x}} \left(I- \frac{xx}{\abs{x}^2} \right),
\end{equation*}   
where $I$ denotes the identity matrix and $xx$ is the dyadic product of $x$ and $x.$

Before handling the integral identities, we need some notation. 
We denote the radial derivative and the angular gradient of a generic function $\phi\colon \R^d \to \R$ by
\begin{equation*}
	\partial_r \phi(x):= \frac{x}{\abs{x}} \cdot \nabla \phi(x), \qquad \nabla_\tau \phi(x):=\left(I- \frac{xx}{\abs{x}^2}\right) \cdot \nabla \phi(x),
\end{equation*}
with these definitions it is easy to see that $\abs{\nabla \phi}^2=\abs{\partial_r \phi}^2+ \abs{\nabla_\tau \phi}^2.$

Now we are in position to deal with the $S$ and $P$ components separately.

\begin{description}[style=unboxed,leftmargin=0cm]
\item[$S$ component.]
We will follow again the strategy in~\cite{F_K_V}: taking the sum~\eqref{first_u_S}$ + k_1^{\frac{1}{2}}$~\eqref{second_u_S}$ + $~\eqref{third_u_S} we have
\begin{equation*} 
	\begin{split}
		k_1\int_{\R^d} \phi_1 \abs{u_S}^2& - \mu \int_{\R^d} \phi_1 \abs{\nabla u_S}^2 
		+ \frac{\mu}{2} \int_{\R^d} \Delta \phi_1 \abs{u_S}^2 + k_1^\frac{1}{2} k_2 \int_{\R^d} \phi_2 \abs{u_S}^2\\
		&+\mu \int_{\R^d} \nabla u_S \cdot D^2 \phi_3\cdot \nabla u_S - \frac{\mu}{4} \int_{\R^d} \Delta^2 \phi_3 \abs{u_S}^2\\
		= \Re \int_{\R^d} \phi_1 u_S f_S &+ k_1^\frac{1}{2} \Im \int_{\R^d} \phi_2 u_S f_S
		- \Re \int_{\R^d} \nabla \phi_3 \cdot \nabla u_S f_S - \frac{1}{2} \Re \int_{\R^d} \Delta \phi_3 u_S f_S.
	\end{split}
\end{equation*}
For a moment, we only consider the first two terms and the last one of the first row and the first term of the second row, using also the fact that the $\phi_i,$ for $i=1,2,3,$ are radial, one has
\begin{equation*}
	\begin{split}
		&k_1\int_{\R^d} \phi_1 \abs{u_S}^2 - \mu \int_{\R^d} \phi_1 \abs{\nabla u_S}^2 + k_1^\frac{1}{2} k_2 \int_{\R^d} \phi_2 \abs{u_S}^2 
		+\mu \int_{\R^d} \nabla u_S \cdot D^2 \phi_3\cdot \nabla u_S\\
		&=\int_{\R^d} \abs{u_S}^2 (k_1 \psi_1 + k_1^\frac{1}{2} k_2 \psi_2) + \mu \int_{\R^d} \abs{\partial_r u_S}^2 (\psi_3'' - \psi_1) 
		+ \mu \int_{\R^d} \abs{\nabla_\tau u_S}^2 \left( \frac{\psi_3'}{\abs{x}} - \psi_1 \right).
	\end{split}
\end{equation*}  
Using the latter in the former and modifying the order of the terms, we obtain
\begin{gather*}
	\mu \int_{\R^d} \abs{\partial_r u_S}^2 (\psi_3''-\psi_1) 
	+\mu \int_{\R^d} \abs{\nabla_\tau u_S}^2 \left( \frac{\psi_3'}{\abs{x}} - \psi_1 \right) 
	+\int_{\R^d} \abs{u_S}^2 (k_1 \psi_1 + k_1^\frac{1}{2} k_2 \psi_2)
	+\mu\int_{\R^d} \abs{u_S}^2 \left(\frac{1}{2} \Delta \phi_1 - \frac{1}{4} \Delta^2 \phi_3 \right)\\
	= \Re \int_{\R^d} \phi_1 u_S f_S 
	+ k_1^\frac{1}{2} \Im \int_{\R^d} \phi_2 u_S f_S -\frac{1}{2} \Re \int_{\R^d} \Delta \phi_3 u_S f_S
	-\Re \int_{\R^d} \nabla \phi_3 \cdot \nabla u_S f_S.    
\end{gather*}
Let us choose $\psi_1$ and $\psi_2$ as functions of $\psi_3,$ to be more precise $\psi_1=\frac{1}{2} \psi_3''$ and $\psi_2 = \frac{1}{\sqrt{\mu}} \sgn(k_2)\,\psi_3',$ the previous identity becomes
\begin{gather*}
	\frac{\mu}{2} \int_{\R^d} \psi_3'' \abs{\partial_r u_S}^2
	+ \mu \int_{\R^d} \abs{\nabla_\tau u_S}^2 \left(\frac{\psi_3'}{\abs{x}} - \frac{\psi_3''}{2} \right)
	+ \frac{k_1}{2}\int_{\R^d} \abs{u_S}^2 \psi_3''\\
	+ \frac{\mu}{4} \int_{\R^d} \abs{u_S}^2 \left( \Delta \phi_3''-\Delta^2 \phi_3 \right)
	+ \frac{1}{\sqrt{\mu}} k_1^\frac{1}{2} \abs{k_2} \int_{\R^d} \psi_3' \abs{u_S}^2\\
	= \frac{1}{2} \Re \int_{\R^d} \psi_3'' u_S f_S 
	+ \frac{1}{\sqrt{\mu}} k_1^\frac{1}{2}  \sgn(k_2)\, \Im \int_{\R^d} \psi_3' u_S f_S 
	- \frac{1}{2} \Re \int_{\R^d} \Delta \phi_3 u_S f_S
	-\Re\int_{\R^d} \nabla \phi_3 \cdot \nabla u_S f_S,  
\end{gather*}
where $\phi_3''(x):=\psi_3''(\abs{x}).$

Choosing $\phi_3(x)=\psi_3(\abs{x}):=\abs{x}^2,$ since the three following identities hold
\begin{equation*}
	\nabla \phi_3(x)=2x, \qquad \Delta \phi_3(x)=2d, \qquad D^2 \phi_3(x)=2 I,
\end{equation*} 
where $I$ denotes the identity matrix, the last becomes
\begin{equation} \label{before_u_S_minus}
	\begin{split}
		\mu \int_{\R^d} \abs{\nabla u_S}^2	+ k_1 \int_{\R^d} \abs{u_S}^2 
		&+ 2\frac{1}{\sqrt{\mu}} k_1^\frac{1}{2} \abs{k_2} \int_{\R^d} \abs{x} \abs{u_S}^2\\
		= (1-d)\,\Re \int_{\R^d} u_S f_S
		+ 2 \frac{1}{\sqrt{\mu}} k_1^\frac{1}{2} \sgn(k_2)&\, \Im \int_{\R^d} \abs{x} u_S f_S 
		-2 \Re \int_{\R^d} x\cdot \nabla u_S f_S.
	\end{split}
\end{equation}
We introduce a vector field $u_S^-$ that is related to $u_S$ by the following definition
\begin{equation} \label{def:u_S_minus}
	u_S^-(x):= \sqrt{\mu}\, e^{-i \frac{1}{\sqrt{\mu}}k_1^\frac{1}{2}\sgn(k_2) \abs{x}}\, u_S(x).
\end{equation}
A straightforward computation shows that 
\begin{equation} \label{gradient_u_S_minus}
	\abs{\nabla u_S^-}^2= \mu \abs{\nabla u_S}^2 + k_1 \abs{u_S}^2,
\end{equation}
using the previous identity, we can sum the first two terms of~\eqref{before_u_S_minus} to obtain
\begin{equation*}
	\int_{\R^d} \abs{\nabla u_S^-} 
	+ 2 \frac{1}{\sqrt{\mu}} k_1^\frac{1}{2} \abs{k_2} \int_{\R^d} \abs{x} \abs{u_S}^2
	= (1-d)\,\Re \int_{\R^d} u_S f_S 
	+ 2 \frac{1}{\sqrt{\mu}}k_1^\frac{1}{2} \sgn(k_2) \, \Im \int_{\R^d} \abs{x} u_S f_S
	-2\Re \int_{\R^d} x\cdot \nabla u_S f_S. 
\end{equation*}
Subtracting from the last identity the~\eqref{first_u_S} with the choice $\phi_1(x):= \frac{1}{\sqrt{\mu}}  \frac{\abs{k_2}}{k_1^{\frac{1}{2}}}\abs{x},$ that is subtracting the following identity
\begin{equation*}
	\frac{1}{\sqrt{\mu}}\, k_1^\frac{1}{2}\abs{k_2} \int_{\R^d} \abs{x} \abs{u_S}^2
	- \sqrt{\mu}\, \frac{\abs{k_2}}{k_1^{\frac{1}{2}}} \, \int_{\R^d} \abs{x} \abs{\nabla u_S}^2 
	+ \sqrt{\mu}\, \frac{d-1}{2}  \frac{\abs{k_2}}{k_1^{\frac{1}{2}}}  \int_{\R^d} \frac{\abs{u_S}^2}{\abs{x}}
	=\frac{1}{\sqrt{\mu}}\,  \frac{\abs{k_2}}{k_1^{\frac{1}{2}}} \, \Re \int_{\R^d} \abs{x} u_S f_S,
\end{equation*}
and making use of the fact that $\Im(z)=-\Re(iz)$ for all $z\in \C,$ one gets
\begin{gather*}
	\int_{\R^d} \abs{\nabla u_S^-}^2 
	+ \frac{1}{\sqrt{\mu}}\frac{\abs{k_2}}{k_1^{\frac{1}{2}}} \int_{\R^d} k_1 \abs{x} \abs{u_S}^2
	+\sqrt{\mu}\, \frac{\abs{k_2}}{k_1^{\frac{1}{2}}} \int_{\R^d} \abs{x} \abs{\nabla u_S}^2
	- \sqrt{\mu}\, \frac{d-1}{2} \frac{\abs{k_2}}{k_1^{\frac{1}{2}}}\int_{\R^d} \frac{\abs{u_S}^2}{\abs{x}}\\
	=(1-d) \Re \int_{\R^d} u_S f_S 
	-2 \Re \int_{\R^d} \abs{x} f_S \left[ \partial_r u_S +i \frac{1}{\sqrt{\mu}} k_1^\frac{1}{2} \sgn(k_2) u_S \right]
	- \frac{1}{\sqrt{\mu}}  \frac{\abs{k_2}}{k_1^{\frac{1}{2}}} \Re \int_{\R^d} \abs{x} u_S f_S. 
\end{gather*}
Using again~\eqref{gradient_u_S_minus} to sum together the last two terms of the first row of the above identity, we have
\begin{equation} \label{fund_u_S}
	\begin{split}
		\int_{\R^d}& \abs{\nabla u_S^-}^2 
		+ \frac{1}{\sqrt{\mu}}\frac{\abs{k_2}}{k_1^\frac{1}{2}} \int_{\R^d} \abs{x} \abs{\nabla u_S^-}^2
		-\sqrt{\mu}\, \frac{d-1}{2}\frac{\abs{k_2}}{k_1^{\frac{1}{2}}} \int_{\R^d} \frac{\abs{u_S}^2}{\abs{x}}\\
		=(1-d)\Re \int_{\R^d}& u_S f_S
		-2\Re\int_{\R^d} \abs{x} f_S 
		\left[ 
					\partial_r u_S + i\frac{1}{\sqrt{\mu}}k_1^\frac{1}{2} \sgn(k_2) u_S  		
		\right]
		-\frac{1}{\sqrt{\mu}} \frac{\abs{k_2}}{k_1^{\frac{1}{2}}} \Re\int_{\R^d} \abs{x} u_S f_S.
	\end{split}
\end{equation}

\begin{remark}
Here and in the sequel, we are going to use the following definition of the $L^2$- norm of a vector field $u\colon \R^d \to \R^d,$ precisely
\begin{equation*}
	\norm{u}:=\norm{u}_{[L^2(\R^d)]^d}= \Bigg( \sum_{j=1}^d \norm{u_j}_{L^2(\R^d)}^2 \Bigg)^\frac{1}{2},
\end{equation*} 
this norm in equivalent to another one, i.e
\begin{equation*}
	\normeq{u}:=\normeq{u}_{[L^2(\R^d)]^d}= \sum_{j=1}^d \norm{u_j}_{L^2(\R^d)},
\end{equation*}
indeed it is very easy to show the two following inequalities
\begin{equation} \label{equivalence_norms}
\norm{u} \leq \normeq{u} \qquad \text{and} \qquad \normeq{u}\leq \sqrt{d} \norm{u}.
\end{equation}
\end{remark}

Now we are in position to estimate from below the left hand side of~\eqref{fund_u_S}, that is the term
\begin{equation*}
	I_S:=\int_{\R^d} \abs{\nabla u_S^-}^2 
			+ \frac{1}{\sqrt{\mu}}\frac{\abs{k_2}}{k_1^\frac{1}{2}} \int_{\R^d} \abs{x} \abs{\nabla u_S^-}^2
			-\sqrt{\mu}\, \frac{d-1}{2}\frac{\abs{k_2}}{k_1^{\frac{1}{2}}} \int_{\R^d} \frac{\abs{u_S}^2}{\abs{x}}. 
\end{equation*}
Using the weighted Hardy inequality
\begin{equation*}
	\forall \psi \in C_0^\infty(\R^d), 
	\qquad \int_{\R^d} \frac{\abs{\psi}^2}{\abs{x}}\leq \frac{4}{(d-1)^2} \int_{\R^d} \abs{x}\abs{\nabla \psi}^2,
\end{equation*} 
and the fact that $\abs{u_S}=\frac{1}{\sqrt{\mu}} \abs{u_S^-},$ we obtain
\begin{equation} \label{bound_I_S}
	I_S\geq \int_{\R^d} \abs{\nabla u_S^-}^2 
			+ \frac{1}{\sqrt{\mu}} \frac{d-3}{d-1} \frac{\abs{k_2}}{k_1^\frac{1}{2}}  \int_{\R^d} \abs{x} \abs{\nabla u_S^-}^2.
\end{equation}
Now, let us bound the three terms on the right of~\eqref{fund_u_S}. Using the following notation 
\begin{equation*}
	I_S^{(1)}:=(1-d) \Re \int_{\R^d} u_S f_S,
\end{equation*}
and exploiting Cauchy-Schwarz and classical Hardy inequalities.
one gets
\begin{equation} \label{bound_I_S^(1)}
	\abs{I_S^{(1)}} \leq \frac{(d-1)}{\sqrt{\mu}} \normeq{\abs{x} f_S} \normeq{\frac{u_S^-}{\abs{x}}}
	\leq \frac{2}{\sqrt{\mu}} \frac{d(d-1)}{(d-2)} \norm{\abs{x} f_S} \norm{\nabla u_S^-}.
\end{equation}
Now we consider the second term
\begin{equation*}
	I_S^{(2)}:= -2 \Re \int_{\R^d} \abs{x} f_S 
	\left[ 
				\partial_r u_S + i\frac{1}{\sqrt{\mu}} k_1^\frac{1}{2} \sgn(k_2)  u_S
	\right],
\end{equation*}  
a straightforward computation shows that 
\begin{equation*} 
	\partial_r u_S^-= \sqrt{\mu}\, e^{-i \frac{1}{\sqrt{\mu}} k_1^\frac{1}{2} \sgn (k_2) \abs{x}} \, 
	\left[\partial_r u_S - i\frac{1}{\sqrt{\mu}} k_1^\frac{1}{2} \sgn(k_2) u_S \right],
\end{equation*}
using the previous identity (actually the conjugated) and using again the Cauchy-Schwarz inequality, we have
\begin{equation} \label{bound_I_S^(2)}
	\abs{I_S^{(2)}}\leq \frac{2}{\sqrt{\mu}} \normeq{\abs{x} f_S} \normeq{\partial_r u_S^-} 
	\leq \frac{2}{\sqrt{\mu}} d\, \norm{\abs{x} f_S} \norm{\nabla u_S^-}.
\end{equation}
If $k_2=0$ we have nothing else to estimate, instead if $k_2\neq 0$ we need also to provide an upper bound for the modulus of the third term 
\begin{equation*}
I_S^{(3)}:=-\frac{1}{\sqrt{\mu}} \frac{\abs{k_2}}{k_1^{\frac{1}{2}}} \Re\int_{\R^d} \abs{x} u_S f_S.
\end{equation*}
In order to estimate this term we need a bound of the $L^2$-norm of $u_S.$
Taking the~\eqref{second_u_S} with the choice $\phi_2:= \frac{k_2}{\abs{k_2}}$ we obtain
\begin{equation*}
	\abs{k_2} \int_{\R^d} \abs{u_S}^2= \frac{k_2}{\abs{k_2}} \Im \int_{\R^d} u_S f_S,
\end{equation*} 
from the last identity it is easy to obtain the estimate we are looking for, precisely one gets
\begin{equation} \label{L_2_bound}
	\norm{u_S}^2\leq \abs{k_2}^{-1} \int_{\R^d} \abs{u_S} \abs{f_S}.
\end{equation}
As a consequence of~\eqref{L_2_bound}, since $\abs{k_2}\leq k_1,$ we can estimate the last term in~\eqref{fund_u_S} as follows:
\begin{equation} \label{bound_I_S^(3)} 
	\begin{split}
		\abs{I_S^{(3)}}&\leq \frac{1}{\sqrt{\mu}} \frac{\abs{k_2}}{k_1^\frac{1}{2}} \normeq{\abs{x} f_S}\normeq{u_S}
		\leq \frac{d}{\sqrt{\mu}} \frac{\abs{k_2}}{k_1^\frac{1}{2}} \norm{\abs{x} f_S}\norm{u_S}
		\leq \frac{d}{\sqrt{\mu}} \norm{\abs{x} f_S} \left( \int_{\R^d} \abs{u_S} \abs{f_S}\right)^\frac{1}{2}\\
		&\leq \frac{d^\frac{3}{2}}{\sqrt{\mu}^\frac{3}{2}} \norm{\abs{x} f_S}^\frac{3}{2} \norm*{\frac{u_S^-}{\abs{x}}}^\frac{1}{2}
		\leq \frac{\sqrt{2}}{\sqrt{\mu}^\frac{3}{2}} \frac{d^\frac{3}{2}}{\sqrt{d-2}} \norm{\abs{x} f_S}^\frac{3}{2} \norm{\nabla u_S^-}^\frac{1}{2}.
	\end{split}
\end{equation}
Using our notation,~\eqref{fund_u_S} can be re-written as
\begin{equation*}
I_S=I_S^{(1)} + I_S^{(2)} + I_S^{(3)},
\end{equation*}
applying~\eqref{bound_I_S},~\eqref{bound_I_S^(1)},~\eqref{bound_I_S^(2)} and~\eqref{bound_I_S^(3)} in the previous identity we have
\begin{equation} \label{preliminary_u_S}
	\norm{\nabla u_S^-}^2 
	+ \frac{1}{\sqrt{\mu}}\,\frac{d-3}{d-1} \frac{\abs{k_2}}{k_1^\frac{1}{2}} \int_{\R^d} \abs{x} \abs{\nabla u_S^-}^2
	\leq \frac{2}{\sqrt{\mu}} \frac{d(2d-3)}{d-2} \norm{\abs{x} f_S} \norm{\nabla u_S^-}
	+ \frac{\sqrt{2}}{\sqrt{\mu}^\frac{3}{2}} \frac{d^\frac{3}{2}}{\sqrt{d-2}} \norm{\abs{x} f_S}^\frac{3}{2} \norm{\nabla u_S^-}^\frac{1}{2}.
\end{equation} 

Now we are in position to use the analogous technique to obtain the same estimate for the $P$ component. Clearly, we mostly omit the details in fact these are the same we have already shown for the divergence free vector field $u_S$.

\item[$P$ component.]
We define the vector field $u_P^-$ in the same way as $u_S^-,$ precisely  
\begin{equation} \label{def:u_P_minus}
	u_P^-(x):=\sqrt{\lambda + 2\mu}\, e^{-i\frac{1}{\sqrt{\lambda +2\mu}} k_1^\frac{1}{2} \sgn(k_2) \abs{x}}\, u_P(x).
\end{equation}
Starting from~\eqref{first_u_P},~\eqref{second_u_P} and~\eqref{third_u_P}, substituting the constant $\mu$ with $\lambda + 2\mu$ in the choices of $\psi_1,\psi_2,\psi_3$ and using the same computations of the $u_S$-case, we have the analogous identity to~\eqref{fund_u_S} for $u_P,$ that is the following
\begin{equation} \label{fund_u_P}
	\begin{split}
		\int_{\R^d} &\abs{\nabla u_P^-}^2
		+ \frac{1}{\sqrt{\lambda + 2\mu}}\frac{\abs{k_2}}{k_1^\frac{1}{2}} \int_{\R^d} \abs{x} \abs{\nabla u_P^-}^2
		-\sqrt{\lambda + 2\mu}\, \frac{d-1}{2}\frac{\abs{k_2}}{k_1^{\frac{1}{2}}} \int_{\R^d} \frac{\abs{u_P}^2}{\abs{x}}\\
		=(1-d)\Re \int_{\R^d} u_P f_P&
		-2\Re\int_{\R^d} \abs{x} f_P \left[ \partial_r u_P + 1\frac{1}{\sqrt{\lambda + 2\mu}}k_1^\frac{1}{2}\sgn(k_2)u_P 
		 \right]
		-\frac{1}{\sqrt{\lambda + 2\mu}} \frac{\abs{k_2}}{k_1^{\frac{1}{2}}} \Re\int_{\R^d} \abs{x} u_P f_P.
	\end{split}
\end{equation}
Again, to save space, let us re-write this identity in this way
\begin{equation} \label{simplified_P} 
I_P=I_P^{(1)} + I_P^{(2)} + I_P^{(3)}.  
\end{equation}
Proceeding in line with the previous case, as a starting point we can estimate from below the left hand side of~\eqref{fund_u_P}, that is the term $I_P,$ and we obtain
\begin{equation} \label{bound_I_P}
	I_P \geq \int_{\R^d} \abs{\nabla u_P^-}^2 
	+ \frac{1}{\sqrt{\lambda + 2\mu}} \frac{d-3}{d-1} \frac{\abs{k_2}}{k_1^\frac{1}{2}}  
	\int_{\R^d} \abs{x} \abs{\nabla u_P^-}^2.
\end{equation}
With respect to the terms $I_P^{(1)},I_P^{(2)},I_P^{(3)}$, one respectively has
\begin{equation} \label{bound_I_P^(1)}
	\abs{I_P^{(1)}} 
	\leq \frac{2}{\sqrt{\lambda + 2\mu}} \frac{d(d-1)}{(d-2)} \norm{\abs{x} f_P} \norm{\nabla u_P^-}.
\end{equation}
\begin{equation} \label{bound_I_P^(2)}
	\abs{I_P^{(2)}}
	\leq \frac{2}{\sqrt{\lambda + 2\mu}} d\, \norm{\abs{x} f_P} \norm{\nabla u_P^-}.
\end{equation} 
\begin{equation} \label{bound_I_P^(3)} 
	\abs{I_P^{(3)}}
	\leq \frac{\sqrt{2}}{\sqrt{\lambda + 2\mu}^\frac{3}{2}} \frac{d^\frac{3}{2}}{\sqrt{d-2}} \norm{\abs{x} f_P}^\frac{3}{2} \norm{\nabla u_P^-}^\frac{1}{2}.
\end{equation}
Applying the estimates~\eqref{bound_I_P},~\eqref{bound_I_P^(1)},~\eqref{bound_I_P^(2)} and~\eqref{bound_I_P^(3)} in~\eqref{simplified_P}, one gets
\begin{equation} \label{preliminary_u_P}
	\begin{split}
		\norm{\nabla u_P^-}^2 
		+ \frac{1}{\sqrt{\lambda + 2\mu}} \frac{d-3}{d-1} \frac{\abs{k_2}}{k_1^\frac{1}{2}} \int_{\R^d} \abs{x} \abs{\nabla u_P^-}^2
		&\leq \frac{2}{\sqrt{\lambda + 2\mu}} \frac{d(2d-3)}{d-2} \norm{\abs{x} f_P} \norm{\nabla u_P^-}\\
		&+ \frac{\sqrt{2}}{\sqrt{\lambda + 2\mu}^\frac{3}{2}} \frac{d^\frac{3}{2}}{\sqrt{d-2}} \norm{\abs{x} f_P}^\frac{3}{2} \norm{\nabla u_P^-}^\frac{1}{2}.
	\end{split}
\end{equation} 
\end{description}
In order to complete the argument, the following elliptic regularity lemma will be useful in the immediate sequel.
\begin{lemma} \label{elliptic_regularity}
	Let $f\in [C^\infty_c(\R^d)]^d$ be a smooth- compactly supported vector field in $\R^d,$ and let $\psi\colon \R^d \to \R$ be a smooth solution to
	\begin{equation} \label{elliptic_problem}
		\Delta \psi= \div f.
	\end{equation}
	Then for any $s\in (-d, d)$ the following estimate holds
	\begin{equation*}
		\norm{\abs{x}^s\nabla \psi}\leq C \norm{\abs{x}^s f},
	\end{equation*} 
	for some constant $C>0$ independent of $f.$
\end{lemma}
Let us introduce a trivial decomposition of our $f:$
\begin{equation*} 
	f=f - \nabla \psi + \nabla \psi,
\end{equation*} 
where $\psi$ is the unique solution of~\eqref{elliptic_problem}; as a consequence we have $\div(f-\nabla \psi)=0.$ By the uniqueness of the Helmholtz decomposition, it follows that $f_S=f-\nabla \psi,$ $f_P= \nabla \psi.$ Substituting these in~\eqref{preliminary_u_S} and~\eqref{preliminary_u_P} respectively, one gets the two following estimates   
\begin{equation*}
	\begin{split}
		\norm{\nabla u_S^-}^2 
		+ \frac{1}{\sqrt{\mu}} \frac{d-3}{d-1} \frac{\abs{k_2}}{k_1^\frac{1}{2}} \int_{\R^d} \abs{x} \abs{\nabla u_S^-}^2
		&\leq \frac{2}{\sqrt{\mu}} \frac{d(2d-3)}{d-2} (\norm{\abs{x} f} + \norm{\abs{x}\nabla \psi} ) \norm{\nabla u_S^-}\\
		&+ \frac{\sqrt{2}}{\sqrt{\mu}^\frac{3}{2}} \frac{d^\frac{3}{2}}{\sqrt{d-2}} (\norm{\abs{x} f} + \norm{\abs{x}\nabla \psi})^\frac{3}{2} \norm{\nabla u_S^-}^\frac{1}{2};
	\end{split}
\end{equation*} 
and
\begin{equation*} 
	\begin{split}
		\norm{\nabla u_P^-}^2 
		+ \frac{1}{\sqrt{\lambda + 2\mu}} \frac{d-3}{d-1} \frac{\abs{k_2}}{k_1^\frac{1}{2}} \int_{\R^d} \abs{x} \abs{\nabla u_P^-}^2
		&\leq \frac{2}{\sqrt{\lambda + 2\mu}} \frac{d(2d-3)}{d-2} \norm{\abs{x} \nabla \psi} \norm{\nabla u_P^-}\\
		&+ \frac{\sqrt{2}}{\sqrt{\lambda + 2\mu}^\frac{3}{2}} \frac{d^\frac{3}{2}}{\sqrt{d-2}} \norm{\abs{x} \nabla \psi}^\frac{3}{2} \norm{\nabla u_P^-}^\frac{1}{2}.
	\end{split}
\end{equation*} 
Using the elliptic regularity result~\ref{elliptic_regularity} we obtain respectively
\begin{equation*}
	\begin{split}
		\norm{\nabla u_S^-}^2 
		+ \frac{1}{\sqrt{\mu}} \frac{d-3}{d-1} \frac{\abs{k_2}}{k_1^\frac{1}{2}} \int_{\R^d} \abs{x} \abs{\nabla u_S^-}^2
		&\leq \frac{2}{\sqrt{\mu}} \frac{d(2d-3)}{d-2} (C+1)\norm{\abs{x} f} \norm{\nabla u_S^-}\\
		&+ \frac{\sqrt{2}}{\sqrt{\mu}^\frac{3}{2}} \frac{d^\frac{3}{2}}{\sqrt{d-2}} (C+1)^\frac{3}{2} \norm{\abs{x} f}^\frac{3}{2} \norm{\nabla u_S^-}^\frac{1}{2};
	\end{split}
\end{equation*} 
and
\begin{equation*}
	\begin{split}
		\norm{\nabla u_P^-}^2 
		+ \frac{1}{\sqrt{\lambda + 2\mu}} \frac{d-3}{d-1} \frac{\abs{k_2}}{k_1^\frac{1}{2}}  
		\int_{\R^d} \abs{x} \abs{\nabla u_P^-}^2
		&\leq \frac{2}{\sqrt{\lambda + 2\mu}} \frac{d(2d-3)}{d-2} C \norm{\abs{x} f} \norm{\nabla u_P^-}\\
		&+ \frac{\sqrt{2}}{\sqrt{\lambda + 2\mu}^\frac{3}{2}} \frac{d^\frac{3}{2}}{\sqrt{d-2}} C^\frac{3}{2} \norm{\abs{x} f}^\frac{3}{2} \norm{\nabla u_P^-}^\frac{1}{2}.
	\end{split}
\end{equation*} 
Recalling that, at the beginning, $f=Vu$ and using~\eqref{smallness_condition_l} one has
\begin{equation*}
	\norm{\abs{x}f}=\norm{\abs{x} V u}\leq \norm{\abs{x} V u_S} 
	+ \norm{\abs{x} V u_P}\leq \frac{\Lambda}{\sqrt{\mu}} \norm{\nabla u_S^-} 
	+ \frac{\Lambda}{\sqrt{\lambda + 2\mu}} \norm{\nabla u_P^-}.
\end{equation*}  
By virtue of the previous inequality and using the convexity of the function $g(x)=\abs{x}^p$ for $p\geq 1$ (in the inequality for the $S$ component), we have
\begin{equation*}
	\begin{split}
		\norm{\nabla u_S^-}^2 
		+ \frac{1}{\sqrt{\mu}}\frac{d-3}{d-1} \frac{\abs{k_2}}{k_1^\frac{1}{2}}  \int_{\R^d} \abs{x} \abs{\nabla u_S^-}^2
		&\leq  \, \frac{2\Lambda}{\mu} \frac{d(2d-3)}{d-2} (C+1)  \norm{\nabla u_S^-}^2 
		+ \frac{4 \Lambda^\frac{3}{2}}{\sqrt{\mu}^3} \frac{d^\frac{3}{2}}{\sqrt{d-2}} (C+1)^\frac{3}{2} \norm{\nabla u_S^-}^2\\
		+ \frac{2 \Lambda}{\sqrt{\mu} \sqrt{\lambda+ 2\mu}} \frac{d(2d-3)}{d-2}(C+1) \norm{\nabla u_S^-}& \norm{\nabla u_P^-}
		+ \frac{4 \Lambda^\frac{3}{2}}{\sqrt{\mu}^\frac{3}{2} \sqrt{\lambda +2\mu}^\frac{3}{2}} \frac{d^\frac{3}{2}}{\sqrt{d-2}} (C+1)^\frac{3}{2} \norm{\nabla u_S^-}^\frac{1}{2} \norm{\nabla u_P^-}^\frac{3}{2};
	\end{split}
\end{equation*} 
and
\begin{equation*}
	\begin{split}
		\norm{\nabla u_P^-}^2& 
		+ \frac{1}{\sqrt{\lambda + 2\mu}} \frac{d-3}{d-1} \frac{\abs{k_2}}{k_1^\frac{1}{2}}  
		\int_{\R^d} \abs{x} \abs{\nabla u_P^-}^2
		\leq \, \frac{2\Lambda}{\lambda + 2\mu} \frac{d(2d-3)}{d-2} C  \norm{\nabla u_P^-}^2 
		+ \frac{4 \Lambda^\frac{3}{2}}{\sqrt{\lambda + 2\mu}^3} \frac{d^\frac{3}{2}}{\sqrt{d-2}} C^\frac{3}{2}  \norm{\nabla u_P^-}^2\\
		&+ \frac{2 \Lambda}{\sqrt{\mu} \sqrt{\lambda+ 2\mu}} \frac{d(2d-3)}{d-2}C \norm{\nabla u_S^-} \norm{\nabla u_P^-}
		+ \frac{4 \Lambda ^\frac{3}{2}}{\sqrt{\mu}^\frac{3}{2} \sqrt{\lambda +2\mu}^\frac{3}{2}} \frac{d^\frac{3}{2}}{\sqrt{d-2}} C^\frac{3}{2} \norm{\nabla u_S^-}^\frac{3}{2} \norm{\nabla u_P^-}^\frac{1}{2}.
	\end{split}
\end{equation*} 
Summing these two inequality together and majoring $C$ with $C+1,$ we obtain
\begin{gather*}
	\norm{\nabla u_S^-}^2 + \norm{\nabla u_P^-}^2 
	+ \frac{1}{\sqrt{\mu}} \frac{d-3}{d-1} \frac{\abs{k_2}}{k_1^\frac{1}{2}}  \int_{\R^d} \abs{x} \abs{\nabla u_S^-}^2
	+ \frac{1}{\sqrt{\lambda + 2\mu}} \frac{d-3}{d-1} \frac{\abs{k_2}}{k_1^\frac{1}{2}} \int_{\R^d} \abs{x} \abs{\nabla u_P^-}^2\\
	\leq \frac{2\Lambda}{\mu} \frac{d(2d-3)}{d-2} (C+1)  \norm{\nabla u_S^-}^2 
	+ \frac{2\Lambda}{\lambda + 2\mu} \frac{d(2d-3)}{d-2} (C+1)  \norm{\nabla u_P^-}^2\\ 
	+ \frac{4 \Lambda^\frac{3}{2}}{\sqrt{\mu}^3} \frac{d^\frac{3}{2}}{\sqrt{d-2}} (C+1)^\frac{3}{2} \norm{\nabla u_S^-}^2
	+ \frac{4 \Lambda^\frac{3}{2}}{\sqrt{\lambda + 2\mu}^3} \frac{d^\frac{3}{2}}{\sqrt{d-2}} (C+1)^\frac{3}{2}  \norm{\nabla u_P^-}^2\\
	+ \frac{4 \Lambda}{\sqrt{\mu} \sqrt{\lambda+ 2\mu}} \frac{d(2d-3)}{d-2} (C+1) \norm{\nabla u_S^-} \norm{\nabla u_P^-}
	+ \frac{4 \Lambda ^\frac{3}{2}}{\sqrt{\mu}^\frac{3}{2} \sqrt{\lambda +2\mu}^\frac{3}{2}} \frac{d^\frac{3}{2}}{\sqrt{d-2}} (C+1)^\frac{3}{2} \norm{\nabla u_S^-}^\frac{1}{2} \norm{\nabla u_P^-}^\frac{3}{2}\\
	+ \frac{4 \Lambda ^\frac{3}{2}}{\sqrt{\mu}^\frac{3}{2} \sqrt{\lambda +2\mu}^\frac{3}{2}} \frac{d^\frac{3}{2}}{\sqrt{d-2}} (C+1)^\frac{3}{2} \norm{\nabla u_S^-}^\frac{3}{2} \norm{\nabla u_P^-}^\frac{1}{2}.
\end{gather*}
Making use of the Young's inequality, which state that for all non-negative real numbers $a$ and $b$ holds
\begin{equation*}
	a b \leq \frac{a^p}{p} + \frac{b^q}{q},
\end{equation*}
where $p, q$ are determined by $\displaystyle\frac{1}{p} + \frac{1}{q}=1,$
one gets
\begin{equation*}
	\norm{\nabla u_S^-}\norm{\nabla u_P^-}\leq \frac{1}{2} \norm{\nabla u_S^-}^2 
	+ \frac{1}{2} \norm{\nabla u_P^-}^2,
\end{equation*}
\begin{equation*}
	\norm{\nabla u_S^-}^\frac{1}{2}\norm{\nabla u_P^-}^\frac{3}{2}\leq \frac{1}{4} \norm{\nabla u_S^-}^2 
	+ \frac{3}{4} \norm{\nabla u_P^-}^2
	\qquad \text{and} \qquad
	\norm{\nabla u_S^-}^\frac{3}{2}\norm{\nabla u_P^-}^\frac{1}{2}\leq \frac{3}{4} \norm{\nabla u_S^-}^2 
	+ \frac{1}{4} \norm{\nabla u_P^-}^2.
\end{equation*}
Using the latter in the former and the fact that $\mu, \lambda +2\mu \geq \min \{\mu, \lambda +2\mu\},$ we have
\begin{gather*}
	\left( 1 
	-\frac{4\Lambda}{\min\{\mu, \lambda + 2\mu\}} \frac{d(2d-3)}{d-2} (C+1)  
	-\frac{8 \Lambda^\frac{3}{2}}{\sqrt{\min\{\mu, \lambda + 2\mu\}}^3} \frac{d^\frac{3}{2}}{\sqrt{d-2}} (C+1)^\frac{3}{2}
	\right)(\norm{\nabla u_S^-}^2 + \norm{\nabla u_P^-}^2)\\
	+\frac{1}{\sqrt{\mu}} \frac{d-3}{d-1} \frac{\abs{k_2}}{k_1^\frac{1}{2}}  \int_{\R^d} \abs{x} \abs{\nabla u_S^-}^2
	+ \frac{1}{\sqrt{\lambda + 2\mu}} \frac{d-3}{d-1} \frac{\abs{k_2}}{k_1^\frac{1}{2}} \int_{\R^d} \abs{x} \abs{\nabla u_P^-}^2\leq 0.
\end{gather*}

Since the two term in the second row are positive, the last inequality becomes
\begin{equation*}
		\left( 
					1 -\frac{4\Lambda}{\min\{\mu, \lambda + 2\mu\}} \frac{d(2d-3)}{d-2} (C+1)  
					- \frac{8 \Lambda^\frac{3}{2}}{\sqrt{\min\{\mu, \lambda + 2\mu\}}^3} \frac{d^\frac{3}{2}}{\sqrt{d-2}} (C+1)^\frac{3}{2}
		\right)(\norm{\nabla u_S^-}^2+\norm{\nabla u_P^-}^2)\leq 0.
\end{equation*}   
Clearly, by virtue of~\eqref{Lambda_condition}, the term in parenthesis is strictly positive then it follows that $u_S^-, u_P^-$ and thus $u_S, u_P$ are identically equal to zero and, as a consequence of the Helmholtz decomposition, $u$ is identically equal to zero too.

We treat now the simpler case $\abs{k_2}>k_1.$
\item[Case $\abs{k_2}>k_1.$]
Let $u\in [H^1(\R^d)]^d$ be a solution of~\eqref{f_equation}, i.e. a solution of~\eqref{decoupling_weak}. In this case only the identities which turn out from the choice of the symmetric multipliers will play a relevant role , this time we don't need to use the anti-symmetric multiplier that, actually, is the less easy to handle. Indeed choosing $v:= \pm u_S$ in the first of~\eqref{decoupling_weak} and $v:=\pm u_P$ in the second of~\eqref{decoupling_weak}, taking real and imaginary parts of the resulting identities and summing these two identities, we obtain respectively for $u_S$ and $u_P$
\begin{equation*}
	(k_1 \pm k_2) \int_{\R^d} \abs{u_S}^2 = \mu \int_{\R^d} \abs{\nabla u_S}^2 
	+ \Re \int_{\R^d} u_S f_S \pm \Im \int_{\R^d} u_S f_S,
\end{equation*}
and
\begin{equation*}
	(k_1 \pm k_2) \int_{\R^d} \abs{u_P}^2 = (\lambda + 2\mu) \int_{\R^d} \abs{\nabla u_P}^2 
	+ \Re \int_{\R^d} u_P f_P \pm \Im \int_{\R^d} u_P f_P.
\end{equation*}
Taking the sum of the previous and making use of the $H^1$- orthogonality of $u_S$ and $u_P,$ one has
\begin{equation} \label{eq_u_P_e_u_S}
	(k_1 \pm k_2) \int_{\R^d} \abs{u}^2= \mu \int_{\R^d} \abs{\nabla u}^2 
	+ (\lambda + \mu) \int_{\R^d} \abs{\nabla u_P}^2 + \Re \int_{\R^d} u f \pm \Im \int_{\R^d} u f. 
\end{equation}
Now we want to estimate the last two terms on the right hand side of~\eqref{eq_u_P_e_u_S}, in order to obtain the  bound we are going to make use only of the Schwarz's inequality, the classical Hardy's inequality~\eqref{classical_Hardy} and the assumption~\eqref{smallness_condition_l}. Indeed, recalling that $f:=V u,$ one has
\begin{equation*}
	\int_{\R^d} \abs{u} \abs{f}\leq \frac{2}{d-2}\Lambda \norm{\nabla u}^2, 
\end{equation*}
using the following trivial chains of inequalities
\begin{equation*}
	\Re\int_{\R^d} u f \geq - \abs*{\int_{\R^d} u f}\geq -\int_{\R^d} \abs{u}\abs{f}, 
	\qquad \text{and} \qquad  
	\pm\Im\int_{\R^d} u f \geq - \abs*{\int_{\R^d} u f}\geq -\int_{\R^d} \abs{u}\abs{f},
\end{equation*}
we easily obtain
\begin{equation*}
	(k_1 \pm k_2) \int_{\R^d} \abs{u}^2 \geq \left(\mu - \frac{4}{d-2} \Lambda \right) \norm{\nabla u}^2 
	+ (\lambda + \mu) \norm{\nabla u_P}^2. 
\end{equation*}
Let us recall that, to make the quadratic form associated to the Lamé operator positive , we have assumed for the Lamé coefficients the condition~\eqref{positive_condition}; under this hypothesis immediately follows that $\lambda + \mu>0$ thus we obtain
\begin{equation*}
	(k_1 \pm k_2) \int_{\R^d} \abs{u}^2 \geq \left(\mu - \frac{4}{d-2} \Lambda \right)\norm{\nabla u}^2. 
\end{equation*}
It's easy to see that any $\Lambda$ verifying~\eqref{Lambda_condition}, necessarily satisfies $\frac{4}{d-2}\Lambda < \mu,$ therefore one gets
\begin{equation*}
	(k_1 \pm k_2) \int_{\R^d} \abs{u}^2 \geq 0.
\end{equation*} 
Thus from the last inequality follows that $k_1 \pm k_2 \geq 0,$ unless $u$ is identically equal to zero.

It is a straightforward exercise to prove that, under conditions~\eqref{smallness_condition_l} and~\eqref{Lambda_condition}, the possible eigenvalue of $-\Delta^\ast + V$ have to be included in the right complex plane, that is $k_1>0.$  
Noticing that we are assuming $\abs{k_2}>k_1>0,$ which implies that the inequality $k_1 \pm k_2 \geq 0$ cannot hold, we obtain $u=0.$ 
\end{description}

%%%%%%%%%%%%%%%%%%%%%%%%%%%%%%%%%%%%%%%%%%%%%%%%%%%%%%%%%%%%%%%%%%%%%%%%%%%%%%%%%%%%%%%%%%%%%%%%%%%%%%%%%%%%%%%%%%%%%%%%%%%%%%%%%%%%%%%%%%%

\section{Uniform resolvent estimate: proof of Theorem~\ref{thm:uniform_resolvent_estimate}}

The aim of this section is to investigate about uniform resolvent estimate for the solution $u\colon \R^d \to \R^d$ of~\eqref{eigenvalue_complete}.

Just to quote a pair of papers on this topic, in a context of Helmholtz equation, we recall Burq, Planchon, Stalker and Tahvildar-Zadeh~\cite{B_P_S_T-Z, B_P_S_T-Z_II} and the work of Barcel\'o, Vega and Zubeldia~\cite{B_V_Z} which generalizes the previous to electromagnetic Hamiltonians. Whereas, for this kind of estimate in an elasticity setting, we can cite~\cite{B_F-G_P-E_R_V}.   

As we have already mentioned in the introduction, we are going to prove a stronger result which establish the validity of a priori estimates; our theorem will follow as a corollary.

In view of the previous comment, we can now start with the proof of~\ref{thm:a_priori_estimates}
\begin{proof}[Proof of Theorem~\ref{thm:a_priori_estimates}]
Since the estimates are different according to the relation between the real and imaginary part of the frequency, that is when $\abs{k_2}\leq k_1$ or the contrary, we treat the two cases separately.

As a starting point, we will easily show that this kind of estimates holds in the free framework, that is in the setting in which $V=0.$ 
Secondly we prove the estimates in the perturbed case, assuming about $V$ the same integral-smallness condition of Theorem~\ref{thm:absence_of_eigenvalues}.

\begin{description}[style=unboxed,leftmargin=0cm]
		\item[Case $\abs{k_2}\leq k_1.$]  
			We consider the case $V=0.$
			In this framework our equation~\eqref{eigenvalue_complete} reduces to the one we considered in Theorem~\ref{thm:absence_of_eigenvalues}, precisely~\eqref{f_equation}. Throughout the proof of Theorem~\ref{thm:absence_of_eigenvalues}, taking into account the Helmholtz decomposition, we proved for this equation the two estimates~\eqref{preliminary_u_S} and~\eqref{preliminary_u_P} respectively for the $S$ and $P$ component of the solution $u$ of~\eqref{f_equation} that, in order to clarify our argument, we are going to rewrite. One had
			\begin{equation} \label{preliminary_u_S_II}
				\norm{\nabla u_S^-}^2 
				+ \frac{1}{\sqrt{\mu}}\,\frac{d-3}{d-1} \frac{\abs{k_2}}{k_1^\frac{1}{2}} \int_{\R^d} \abs{x} \abs{\nabla u_S^-}^2
				\leq \frac{2}{\sqrt{\mu}} \frac{d(2d-3)}{d-2} \norm{\abs{x} f_S} \norm{\nabla u_S^-}
				+ \frac{d^\frac{3}{2}}{\sqrt{\mu}^\frac{3}{2}} \frac{\sqrt{2}}{\sqrt{d-2}} \norm{\abs{x} f_S}^\frac{3}{2} \norm{\nabla u_S^-}^\frac{1}{2},
			\end{equation}
			and
			\begin{equation*}
			  \begin{split}
					\norm{\nabla u_P^-}^2 
					+ \frac{1}{\sqrt{\lambda + 2\mu}} \frac{d-3}{d-1} \frac{\abs{k_2}}{k_1^\frac{1}{2}} \int_{\R^d} \abs{x} \abs{\nabla u_P^-}^2
					&\leq \frac{2}{\sqrt{\lambda + 2\mu}} \frac{d(2d-3)}{d-2} \norm{\abs{x} f_P} \norm{\nabla u_P^-}\\
					&+ \frac{d^\frac{3}{2}}{\sqrt{\lambda + 2\mu}^\frac{3}{2}} \frac{\sqrt{2}}{\sqrt{d-2}} \norm{\abs{x} f_P}^\frac{3}{2} \norm{\nabla u_P^-}^\frac{1}{2}.
				\end{split}
			\end{equation*}
			Let us only consider the first inequality, the details for the second one will be similar.
				
			We want to estimate the right hand side of the inequality, to this end, let $\varepsilon,$ $\delta>0$ and making use of the Young's inequality one has
			\begin{equation*}
				\norm{\abs{x} f_S} \norm{\nabla u_S^-}\leq \frac{1}{2\varepsilon^2} \norm{\abs{x} f_S}^2 +\frac{\varepsilon^2}{2} \norm{\nabla u_S^-}^2
				\qquad \text{and} \qquad
				\norm{\abs{x} f_S}^\frac{3}{2} \norm{\nabla u_S^-}^\frac{1}{2} \leq \frac{3}{4 \delta^\frac{4}{3}} \norm{\abs{x} f_S}^2 + \frac{\delta^4}{4} \norm{\nabla u_S^-}^2.
			\end{equation*}
			Putting this two in~\eqref{preliminary_u_S_II} and observing that the quantity $\frac{1}{\sqrt{\mu}} \frac{\abs{k_2}}{k_1^\frac{1}{2}} \frac{d-3}{d-1} \int_{\R^d} \abs{x} \abs{\nabla u_S^-}^2$ is positive, we get 
			\begin{equation*}
				\begin{split}
					\norm{\nabla u_S^-}^2 
					&\leq \frac{1}{\sqrt{\mu}} \frac{1}{\varepsilon^2} \frac{d(2d-3)}{d-2} \norm{\abs{x} f_S}^2 
					+ \varepsilon^2  \frac{1}{\sqrt{\mu}} \frac{d(2d-3)}{d-2} \norm{\nabla u_S^-}^2
					+ \frac{3 \sqrt{2}}{4 \delta^\frac{4}{3}} \frac{1}{\sqrt{\mu}^\frac{3}{2}} \frac{d^\frac{3}{2}}{\sqrt{d-2}} \norm{\abs{x} f_S}^2\\
					&+ \delta^4 \frac{\sqrt{2}}{4} \frac{1}{\sqrt{\mu}^\frac{3}{2}}  \frac{d^\frac{3}{2}}{\sqrt{d-2}} \norm{\nabla u_S^-}^2.
				\end{split}
			\end{equation*}
			Thus it may be concluded that
			\begin{equation*}
				\left( 1 
							 -\varepsilon^2  \frac{1}{\sqrt{\mu}} \frac{d(2d-3)}{d-2}
							 -\delta^4 \frac{\sqrt{2}}{4} \frac{1}{\sqrt{\mu}^\frac{3}{2}}  \frac{d^\frac{3}{2}}{\sqrt{d-2}}
				\right) 
				\norm{\nabla u_S^-}^2 
				\leq 
				\left(
							\frac{1}{\sqrt{\mu}} \frac{1}{\varepsilon^2} \frac{d(2d-3)}{d-2}
							+ \frac{3 \sqrt{2}}{4 \delta^\frac{4}{3}} \frac{1}{\sqrt{\mu}^\frac{3}{2}} \frac{d^\frac{3}{2}}{\sqrt{d-2}}
				\right)\norm{\abs{x} f_S}^2.
			\end{equation*}
			The same calculations done for the $P$ component give
			\begin{gather*}
				\left( 1 
							 -\varepsilon^2  \frac{1}{\sqrt{\lambda + 2\mu}} \frac{d(2d-3)}{d-2}
							 -\delta^4 \frac{\sqrt{2}}{4} \frac{1}{\sqrt{\lambda + 2\mu}^\frac{3}{2}}  \frac{d^\frac{3}{2}}{\sqrt{d-2}}
				\right) 
				\norm{\nabla u_P^-}^2\\ 
				\leq 
				\left(
							\frac{1}{\sqrt{\lambda + 2\mu}} \frac{1}{\varepsilon^2} \frac{d(2d-3)}{d-2}
							+ \frac{3 \sqrt{2}}{4 \delta^\frac{4}{3}} \frac{1}{\sqrt{\lambda + 2\mu}^\frac{3}{2}} \frac{d^\frac{3}{2}}{\sqrt{d-2}}
				\right)\norm{\abs{x} f_P}^2.
			\end{gather*}
			Now, since $\mu, \lambda+2\mu \geq \min\{\mu, \lambda + 2\mu\}$ and  choosing $\varepsilon, \delta$ small enough, one can write
			\begin{equation*}
				\norm{\nabla u_S^-} \leq D_{\varepsilon, \delta} \norm{\abs{x} f_S},
			\end{equation*}
			and 
			\begin{equation*}
				\norm{\nabla u_P^-} \leq D_{\varepsilon, \delta}\norm{\abs{x} f_P},
			\end{equation*}
			where 
			\begin{equation*}
				D_{\varepsilon, \delta}= \left( 
				\frac{
							\frac{1}{\sqrt{\min\{\mu, \lambda + 2\mu\}}} \frac{1}{\varepsilon^2} \frac{d(2d-3)}{d-2}
							+\frac{3 \sqrt{2}}{4 \delta^\frac{4}{3}} \frac{1}{\sqrt{\min\{\mu, \lambda + 2\mu\}}^\frac{3}{2}} \frac{d^\frac{3}{2}}{\sqrt{d-2}}
							}
							{
							1 
							-\varepsilon^2  \frac{1}{\sqrt{\min\{\mu, \lambda + 2\mu\}}} \frac{d(2d-3)}{d-2}
							-\delta^4 \frac{\sqrt{2}}{4} \frac{1}{\sqrt{\min\{\mu, \lambda + 2\mu\}}^\frac{3}{2}}  \frac{d^\frac{3}{2}}{\sqrt{d-2}}
							}
				\right)^\frac{1}{2}.
			\end{equation*}
			At the end, using the trivial Helmholtz decomposition of $f=f-\nabla \psi + \nabla \psi$ and the elliptic regularity Lemma~\ref{elliptic_regularity}, one easily concludes
				\begin{equation*}
					\norm{\nabla u_S^-} \leq c \norm{\abs{x} f} 
					\qquad \text{and} \qquad
					\norm{\nabla u_P^-} \leq c \norm{\abs{x} f},
				\end{equation*}
				where $c:= (C+1) D_{\varepsilon, \delta}.$ Let us remark that $c>0$ does not depend on the frequency $k$ and on $f.$ 
				
				Now we can prove our result in the perturbed setting, i.e. $V\neq 0.$
					
				First of all we define $g:=Vu$ and $h:=f+g.$ Thus, with this notation, $u$ solves the following eigenvalues equation
				\begin{equation} \label{h_problem}
					\Delta^\ast u + ku=h;
				\end{equation}
				again we have these estimates for the two components of the solution:
				\begin{equation*}
					\norm{\nabla u_S^-}^2 
					\leq \frac{2}{\sqrt{\mu}} \frac{d(2d-3)}{d-2} \norm{\abs{x} h_S} \norm{\nabla u_S^-}
					+ \frac{d^\frac{3}{2}}{\sqrt{\mu}^\frac{3}{2}} \frac{\sqrt{2}}{\sqrt{d-2}} \norm{\abs{x} h_S}^\frac{3}{2} \norm{\nabla u_S^-}^\frac{1}{2},
				\end{equation*}
				and
				\begin{equation*}
					\norm{\nabla u_P^-}^2 
					\leq \frac{2}{\sqrt{\lambda + 2\mu}} \frac{d(2d-3)}{d-2} \norm{\abs{x} h_P} \norm{\nabla u_P^-}
					+ \frac{d^\frac{3}{2}}{\sqrt{\lambda + 2\mu}^\frac{3}{2}} \frac{\sqrt{2}}{\sqrt{d-2}} \norm{\abs{x} h_P}^\frac{3}{2} \norm{\nabla u_P^-}^\frac{1}{2}.
				\end{equation*}
				We only consider the first one. Clearly, since $h_S=f_S + g_S,$ it can be rewritten as
				\begin{equation*}
					\begin{split}
						\norm{\nabla u_S^-}^2 
						&\leq \frac{2}{\sqrt{\mu}} \frac{d(2d-3)}{d-2} \norm{\abs{x} f_S} \norm{\nabla u_S^-}
						+ \frac{2}{\sqrt{\mu}^\frac{3}{2}} \frac{d^\frac{3}{2}}{\sqrt{d-2}} \norm{\abs{x} f_S}^\frac{3}{2} \norm{\nabla u_S^-}^\frac{1}{2}\\
						&+\frac{2}{\sqrt{\mu}} \frac{d(2d-3)}{d-2} \norm{\abs{x} g_S} \norm{\nabla u_S^-}
						+ \frac{2}{\sqrt{\mu}^\frac{3}{2}} \frac{d^\frac{3}{2}}{\sqrt{d-2}} \norm{\abs{x} g_S}^\frac{3}{2} \norm{\nabla u_S^-}^\frac{1}{2}.
					\end{split}
				\end{equation*}
				Since in the free case we have already bounded the terms in which $f$ appears, now let us only consider the terms involving $g$.
				We introduce the trivial decomposition of $g=g-\nabla \phi + \nabla \phi,$ where, as usual, $\phi$ is the unique solution of the elliptic problem $\Delta \phi=\div g.$ Following the strategy in the Theorem~\ref{thm:absence_of_eigenvalues} about the absence of eigenvalues and, in particular, recalling that formerly $g=Vu$ and that $V$ satisfies~\eqref{smallness_condition_l}, one can show 
				\begin{gather*}
					\norm{\nabla u_S^-}^2 
					\leq \frac{2}{\sqrt{\mu}} \frac{d(2d-3)}{d-2} \norm{\abs{x} f_S} \norm{\nabla u_S^-}
					+ \frac{2}{\sqrt{\mu}^\frac{3}{2}} \frac{d^\frac{3}{2}}{\sqrt{d-2}} \norm{\abs{x} f_S}^\frac{3}{2} \norm{\nabla u_S^-}^\frac{1}{2}\\
					+\frac{2 \Lambda}{\mu} \frac{d(2d-3)}{d-2} (C+1) \norm{\nabla u_S^-}^2
					+ \frac{2\sqrt{2} \Lambda^\frac{3}{2}}{\sqrt{\mu}^3} \frac{d^\frac{3}{2}}{\sqrt{d-2}} (C+1)^\frac{3}{2} \norm{\nabla u_S^-}^2\\
					+\frac{2 \Lambda}{\sqrt{\mu} \sqrt{\lambda + 2\mu}} \frac{d (2d-3)}{d-2} (C+1) \norm{\nabla u_S^-} \norm{\nabla u_P^-}
					+\frac{2\sqrt{2} \Lambda^\frac{3}{2}}{\sqrt{\mu}^\frac{3}{2} \sqrt{\lambda + 2\mu}^\frac{3}{2}} \frac{d^\frac{3}{2}}{\sqrt{d-2}} (C+1)^\frac{3}{2} \norm{\nabla u_S^-}^\frac{1}{2} \norm{\nabla u_P^-}^\frac{3}{2}.
				\end{gather*}
				For the $P$ component we have the following analogue estimate
				\begin{gather*}
					\norm{\nabla u_P^-}^2 
					\leq \frac{2}{\sqrt{\lambda + 2\mu}} \frac{d(2d-3)}{d-2} \norm{\abs{x} f_P} \norm{\nabla u_P^-}
					+ \frac{2}{\sqrt{\lambda + 2\mu}^\frac{3}{2}} \frac{d^\frac{3}{2}}{\sqrt{d-2}} \norm{\abs{x} f_P}^\frac{3}{2} \norm{\nabla u_P^-}^\frac{1}{2}\\
					+\frac{2 \Lambda}{\lambda + 2\mu} \frac{d(2d-3)}{d-2} C \norm{\nabla u_P^-}^2
					+ \frac{2\sqrt{2} \Lambda^\frac{3}{2}}{\sqrt{\lambda + 2\mu}^3} \frac{d^\frac{3}{2}}{\sqrt{d-2}} C^\frac{3}{2} \norm{\nabla u_P^-}^2\\
					+\frac{2 \Lambda}{\sqrt{\mu} \sqrt{\lambda + 2\mu}} \frac{d (2d-3)}{d-2} C \norm{\nabla u_S^-} \norm{\nabla u_P^-}
					+\frac{2\sqrt{2} \Lambda^\frac{3}{2}}{\sqrt{\mu}^\frac{3}{2} \sqrt{\lambda + 2\mu}^\frac{3}{2}} \frac{d^\frac{3}{2}}{\sqrt{d-2}} C^\frac{3}{2} \norm{\nabla u_S^-}^\frac{3}{2} \norm{\nabla u_P^-}^\frac{1}{2}.
				\end{gather*}
				Now estimating the terms involving $f$ as in the free case, summing these inequalities, and using the Young's inequality, we obtain
				\begin{gather*}
					\left(
								1
								-\frac{4\Lambda}{\min\{\mu,\, \lambda+ 2\mu\}} \frac{d(2d-3)}{d-2} (C+1)
								-\frac{4\sqrt{2} \Lambda^\frac{3}{2}}{\sqrt{\min\{\mu,\, \lambda+ 2\mu\}}^3}\frac{d^\frac{3}{2}}{\sqrt{d-2}} (C+1)^\frac{3}{2} 		\right. \\
					\left.
								-\varepsilon^2 \frac{1}{\sqrt{\min\{\mu,\, \lambda+ 2\mu\}}} \frac{d(2d-3)}{d-2}
								-\delta^4 \frac{\sqrt{2}}{4} \frac{1}{\sqrt{\min\{\mu,\, \lambda+ 2\mu\}}^\frac{3}{2}}  \frac{d^\frac{3}{2}}{\sqrt{d-2}}
					\right) 
					(\norm{\nabla u_S^-}^2 +\norm{\nabla u_P^-}^2 ) \\
					\leq 
					\left(
								\frac{1}{\sqrt{\mu}} \frac{1}{\varepsilon^2} \frac{d(2d-3)}{d-2}
								+ \frac{3 \sqrt{2}}{4 \delta^\frac{4}{3}} \frac{1}{\sqrt{\mu}^\frac{3}{2}} \frac{d^\frac{3}{2}}{\sqrt{d-2}}
					\right)
					(\norm{\abs{x} f_S}^2 + \norm{\abs{x} f_P}^2)
				\end{gather*}
				
				Since $V$ satisfies~\eqref{smallness_condition_l} and assuming $\varepsilon, \delta$ sufficiently small, the constant in the left hand side of the previous inequality is positive, thus we can write
				\begin{equation*}
				 (\norm{\nabla u_S^-}^2 + \norm{\nabla u_P^-}^2) \leq D^2_{\varepsilon, \delta} (\norm{\abs{x} f_S}^2 + \norm{\abs{x} f_P}^2),
				\end{equation*}
				where, obviously, $D^2_{\varepsilon, \delta}$ is the ratio between the two constants which respectively appear on the right and on the left hand side of the last but one inequality.
								
				Using now the trivial Helmholtz decomposition of $f=f-\nabla \psi + \nabla \psi$ and the elliptic regularity Lemma~\ref{elliptic_regularity}, one easily has
				\begin{equation*}
					\norm{\nabla u_S^-}^2 + \norm{\nabla u_P^-}^2 \leq c^2  \norm{\abs{x} f}^2,
				\end{equation*}
				where
				\begin{equation*}
					c^2:= 2 (C+1)^2 D^2_{\varepsilon, \delta}
				\end{equation*}
				does not depend on the frequency $k$ and on $f.$
				Moreover, it is clear that the following hold
				\begin{equation*}
					\norm{\nabla u_S^-} \leq c \norm{\abs{x} f} 
					\qquad \text{and} \qquad
					\norm{\nabla u_P^-} \leq c \norm{\abs{x} f}.
				\end{equation*}
					
				Now we can treat the less technical case. 
			\item[Case $\abs{k_2}>k_1.$]
				
				First we consider the free setting.
				
				As the previous case, our equation~\eqref{eigenvalue_complete} becomes the one we have considered in Theorem~\ref{thm:absence_of_eigenvalues}, precisely~\eqref{f_equation}.
				Choosing $\phi_1=1$ in~\eqref{first_u_S} and in~\eqref{first_u_P} one obtains respectively
				\begin{equation*}
					k_1 \int_{\R^d} \abs{u_S}^2 - \mu \int_{\R^d} \abs{\nabla u_S}^2 = \Re \int_{\R^d} u_S f_S
				\end{equation*}
				and 
				\begin{equation*}
					k_1 \int_{\R^d} \abs{u_P}^2 - (\lambda + 2\mu) \int_{\R^d} \abs{\nabla u_P}^2 = \Re \int_{\R^d} u_P f_P.
				\end{equation*}
				Taking the sum of the previous and making use of the $H^1$- orthogonality of $u_S$ and $u_P,$ one has
				\begin{equation*}
				k_1 \int_{\R^d} \abs{u}^2 
				- \mu \int_{\R^d} \abs{\nabla u}^2
				- (\lambda + 2\mu) \int_{\R^d} \abs{\nabla u_P}^2
				=\Re \int_{\R^d} u f. 
				\end{equation*}
				Starting from the identities~\eqref{L_2_bound} and the analogue for $u_P$ it is not difficult to prove the same estimate for $u,$ precisely one obtain
				\begin{equation*}
					\abs{k_2} \int_{\R^d} \abs{u}^2 \leq \int_{\R^d} \abs{u} \abs{f}.
				\end{equation*}
				 Using the latter in the former (here we need the assumption $\abs{k_2}>k_1$) and observing the positivity of the term $(\lambda + 2\mu) \int_{\R^d} \abs{\nabla u_P}^2,$ we have
				\begin{equation*}
					\mu \int_{\R^d} \abs{\nabla u}^2\leq 2 \int_{\R^d} \abs{u} \abs{f}.
				\end{equation*}
				From the Cauchy Schwarz and the Hardy inequalities follows
				\begin{equation*}
					\mu \norm{\nabla u}^2 \leq \frac{4d}{d-2} \norm{x f} \norm{\nabla u}.
				\end{equation*}
				Thus, it may be concluded that
				\begin{equation*}
					\norm{\nabla u}< \frac{1}{\mu} \frac{4d}{d-2} \norm{\abs{x} f}.
				\end{equation*}
				
				We now proceed to show the a priori estimates in the perturbed context.
				Exploiting the same notation we have used in the case $\abs{k_2}\leq k_1,$ again $u$ solves the equation~\eqref{h_problem}. As a consequence of the estimates we have just proved for the free case, recalling that $h= f+ g,$ one easily obtains 
				\begin{equation*}
					\norm{\nabla u} <\frac{1}{\mu} \frac{4d}{d-2} \norm{\abs{x} h} 
					\leq \frac{1}{\mu} \frac{4d}{d-2} \norm{\abs{x} f}
					+ \frac{1}{\mu} \frac{4d}{d-2} \norm{\abs{x} g}.
				\end{equation*}
				Writing now explicitly $g$ as $Vu,$ by assumption~\eqref{smallness_condition_l} we have
				\begin{equation*}
					\norm{\nabla u} <\frac{1}{\mu} \frac{4d}{d-2} \norm{\abs{x} f} 
					+\frac{\Lambda}{\mu} \frac{4d}{d-2} \norm{\nabla u} 
				\end{equation*}
				or, more explicitly
				\begin{equation*}
					\left( 
							1-\frac{\Lambda}{\mu} \frac{4d}{d-2}
					\right) \norm{\nabla u} 
					< \frac{1}{\mu} \frac{4d}{d-2} \norm{\abs{x} f}.
				\end{equation*}
				The condition~\eqref{Lambda_condition} about $\Lambda$ guarantee the positivity of the parenthesis of the left hand side and then the theorem is proved.
		\end{description}	
		\end{proof}
		
	Finally, we are in position to prove the uniform resolvent estimate we are looking for.
	\begin{proof}[Proof of Theorem~\eqref{thm:uniform_resolvent_estimate}]
		First of all, we consider the case $\abs{k_2} \leq k_1,$ as a consequence of~\eqref{a_priori_u_S_minus_u_P_minus}, making use of the Hardy's inequality, it is not difficult to show that the following chain of inequalities holds
	\begin{equation*}
		\begin{split}
			\norm{\abs{x}^{-1} u} &\leq \frac{1}{\sqrt{\mu}} \norm{\abs{x}^{-1} u_S^-} + \frac{1}{\sqrt{\lambda + 2\mu}} \norm{\abs{x}^{-1} u_P^-}
														\leq \frac{2}{d-2} \frac{1}{\sqrt{\min\{\mu, \lambda+ 2\mu\}}} (\norm{\nabla u_S^-} + \norm{\nabla u_P^-})\\
														&\leq \frac{4c}{d-2} \frac{1}{\sqrt{\min\{\mu, \lambda+ 2\mu\}}} \norm{\abs{x} f}.
		\end{split}												
	\end{equation*}
	Assuming $\abs{k_2} > k_1,$ using~\eqref{a_priori_u} and again the Hardy's inequality, we have 
	\begin{equation*}
		\norm{\abs{x}^{-1} u}\leq \frac{2c}{d-2}  \norm{\abs{x} f}.
	\end{equation*}
\end{proof}

\end{document}